\documentclass{amsart}

\usepackage{amsmath}
\usepackage{amssymb} 
\usepackage{epsbox,labelfig}

\usepackage{graphicx}

\newtheorem{thm}{Theorem}[section]
\newtheorem{cor}[thm]{Corollary}
\newtheorem{lem}[thm]{Lemma}
\newtheorem{prop}[thm]{Proposition}
\newtheorem{claim}{Claim}
\theoremstyle{definition}

\newcommand{\Z}{ \mathbb{Z}}
 
  \newcommand{\aut}{ \textrm{Aut}\,}
 \newcommand{\red}{\textrm{red}\,}
  \newcommand{\Hom}{\textrm{Hom}}
    
 \newcommand{\barV}{\overline{V}\,}
 \newcommand{\barW}{\overline{W}\,}
  
\title{The classification of Wada-type representations of braid groups}
\author{Tetsuya Ito}
\address{Graduate School of Mathematical Science, University of Tokyo, Japan}
\email{tetitoh@ms.u-tokyo.ac.jp}
\urladdr{http://ms.u-tokyo.ac.jp/~tetitoh}

\begin{document}

\begin{abstract}
We give a classification of Wada-type representations of the braid groups, and solutions of a variant of the set-theoretical Yang-Baxter equation adapted to the free-product group structure. As a consequence, we prove Wada's conjecture: There are only seven types of Wada-type representations up to certain symmetries.
\end{abstract}
\maketitle
 
\section{Introduction}

Let $B_{n}$ be the braid group of $n$-strands defined by the presentation
\[ B_{n} = \left\langle \sigma_{1},\ldots,\sigma_{n-1} \: 
\begin{array}{|ll}
\sigma_{i}\sigma_{j}\sigma_{i}=\sigma_{j}\sigma_{i}\sigma_{j}, & |i-j|=1 \\
\sigma_{i}\sigma_{j}=\sigma_{j}\sigma_{i}, & |i-j|>1
\end{array}
 \right\rangle \]
and $\aut(F_{n})$ be the automorphism group of $F_{n}$, the free group of rank $n$.
Throughout the paper, we always consider {\em left} actions, hence we adapt the convention that $\aut(F_{n})$ acts on $F_{n}$ from {\em left}.

The braid group $B_{n}$ is naturally identified with $MCG(D_{n}, \partial D_{n})$, the relative mapping class group of the $n$-punctured disc $D_{n} = D^{2}-\{n \textrm {-points}\}$.
 The {\em Artin representation} is a homomorphism $\Phi: B_{n} \rightarrow \aut(F_{n})$ induced by the left action of $B_{n}$ on the group $\pi_{1}(D_{n}) = F_{n}$.
 
For a suitable choice of free generators $x_{1},\ldots,x_{n}$ of $F_{n}$, $\Phi$ is written as
\[ [\Phi(\sigma_{i})](x_{j}) = \left\{ \begin{array}{ll}
x_{i+1} & (j=i)\\
x_{i+1}^{-1}x_{i}x_{i+1} & (j=i+1)\\
x_{j} & (j \neq i,i+1).
\end{array}\right.\]

Now we consider a generalization of the Artin representation.
Let $\tau: F_{2} \rightarrow F_{2}$ be an automorphism of the free group of rank two generated by $x$ and $y$.
Such an automorphism is determined by $ \tau(x) = W(x,y) $ and $ \tau(y)=V(x,y)$, so we will write $\tau = \tau_{W,V}$ by using a pair of reduced words $(W=W(x,y),V=V(x,y))$ on $\{x^{\pm 1}, y^{\pm 1}\}$ .
For $1\leq i \leq n-1$, let $\tau_{i}$ be an automorphism of $F_{n}$ defined by
\[\tau_{i} = id_{F_{i-1}}*\tau*id_{F_{n-i-1}}: F_{n}= F_{i-1}*F_{2}*F_{n-i-1} \rightarrow F_{i-1}*F_{2}*F_{n-i-1}=F_{n}. \]
By using free generators $x_{1},\ldots,x_{n}$ of $F_{n}$, $\tau_{i}$ is given as
\[ \tau_{i}(x_{j}) = \left\{ \begin{array}{ll}
W(x_{i}, x_{i+1}) & (j=i)\\
V(x_{i}, x_{i+1}) & (j=i+1)\\
x_{j} & (j \neq i,i+1).
\end{array}\right.\]
 We say a representation of the braid group $\rho:B_{n} \rightarrow \aut(F_{n})$ is a {\em Wada-type representation} if $\rho(\sigma_{i})=\tau_{i}$ for some automorphism $\tau = \tau_{W,V}$. In this situation we say a pair of reduced words $(W,V)$ {\em defines} a Wada-type representation.
The Artin representation is a Wada-type representation defined by $(y,y^{-1}xy)$.

In \cite{w}, Wada initiated the study of Wada-type representations and gave a list of Wada-type representations by using computer experiments. Wada's motivation was to construct group-valued invariants of links, obtained by imitating the presentation of the link groups via the Artin representation \cite{b}, \cite{w}.
Recently, it is observed that Wada-type representations are closely related to (bi)racks and (bi)quandles, the algebraic objects encoding the combinatorics of (virtual) knot diagrams \cite{fjk}. Moreover, Wada's group invariants are generalized by Crisp-Paris \cite{cp} using general groups, and by the author \cite{i} using quandles. Thus, it is interesting to look for new examples of Wada-type representations.

The aim of this paper is to give a classification of Wada-type representations of the braid groups. We show that Wada's list in \cite{w} is complete, so modulo some natural symmetries, there are essentially only seven types of Wada-type representations. Actually we will classify more general objects: The solutions of a group-theoretical variant of the set-theoretical Yang-Baxter equation.

Recall that the set-theoretical Yang-Baxter equation is an equation in the monoidal category of sets, where the monoidal structure is given by the Cartesian product.
For a set $X$ and $R \in Map(X\times X, X\times X)$, the set-theoretical Yang-Baxter equation is an equation in $Map(X\times X \times X, X\times X\times X)$ given as
\[ R^{12}R^{13}R^{23} = R^{23}R^{13}R^{12} \]
where $R^{ij} \in Map(X\times X \times X, X\times X\times X)$ represents the $R$ action on $i$-th and $j$-th components of $X\times X\times X$.
 
Instead of the monoidal category of sets, we use the monoidal category of groups, where the monoidal structure is given by the free product.
For a group $G$ and $R \in Hom(G*G, G*G)$, we study the same equation 
\[ R^{12}R^{13}R^{23} = R^{23}R^{13}R^{12} \]
in $\Hom (G*G*G, G*G*G)$,
where $R^{ij} \in \Hom (G*G*G, G*G*G)$ represents the $R$ action on $i$-th and $j$-th components of $G*G*G$.
We call this equation the {\em free-product group-theoretical Yang-Baxter equation}, ({\em FGYBE}, in short).
We will classify the solutions of FGYBE for the case $G=\Z$, the infinite cyclic group.

By definition, $\tau \in \Hom(\Z*\Z,\Z*\Z) = \Hom(F_{2},F_{2})$ is a solution of FGYBE if and only if the equality
\[ (id * \tau)(\tau * id) (id* \tau) = (\tau * id) (id * \tau)(\tau * id).\]
holds in $\Hom(\Z*\Z*\Z,\Z*\Z*\Z) =\Hom(F_{3},F_{3})$.

For reduced words $W,V$ on $\{x^{\pm 1}, y^{\pm 1}\}$, let $\tau_{W,V} \in \Hom(F_{2})$ be a homomorphism defined by $\tau(x)= W(x,y)$, $\tau(y)= V(x,y)$. 
We say a pair of reduced words $(W,V)$ is a solution of FGYBE if $\tau_{W,V}$ is a solution of FGYBE. 
By definition $(W,V)$ defines a Wada-type representation if and only if $(W,V)$ is an invertible solution of FGYBE. 

Let $B_{3}^{+}$ be the 3-strand positive braid monoid. Then $(W,V)$ is a solution of FGYBE if and only if the map $\Phi: B_{3}^{+} \rightarrow \Hom(F_{3},F_{3})$ defined by $\Phi(\sigma_{1}) = \tau_{W,V} * id$, $\Phi(\sigma_{2}) = id * \tau_{W,V}$ is a monoid homomorphism.
In this point of view, it is convenient to express $\tau_{W,V}$ diagrammatically as in Figure \ref{fig:action}. 

\begin{figure}[htbp]
 \begin{center}
\SetLabels
(0.25*0.08) $x$\\
(0.7*0.08) $y$\\
(0.25*0.9) $W(x,y)$\\
(0.7*0.9) $V(x,y)$\\
  \endSetLabels
\strut\AffixLabels{\includegraphics*[scale=0.5, width=40mm]{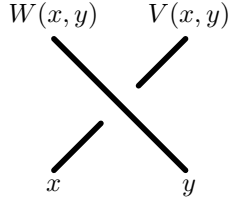}}
 \caption{Diagrammatic expression of $\tau_{W,V}$}
 \label{fig:action}
  \end{center}
\end{figure}

By direct computations, $(W,V)$ is a solution of FGYBE if and only if $(W,V)$ satisfies the following three conditions [T],[M] and [B].
\begin{align*}
\text{[T]} & \quad W(W(x,y),W(V(x,y),z)) = W(x,W(y,z)). \\
\text{[M]} & \quad V(W(x,y), W(V(x,y),z)) = W( V(x,W(y,z)), V(y,z)). \\
\text{[B]}  & \quad  V(V(x,y),z) = V(V(x,W(y,z)), V(y,z)).
\end{align*}
 These three conditions are expressed diagrammatically by Figure \ref{fig:axiom} where the conditions [T],[M], and [B] represent the compatibility conditions for the labellings on {\em top (left-most)}, {\em middle}, and {\em bottom (right-most)} strands respectively. 
 
 \begin{figure}[htbp]
 \begin{center}
\SetLabels
(0.01*0.05) $x$\\
(0.21*0.05) $y$\\
(0.4*0.05) $z$\\
(0.6*0.05) $x$\\
(0.79*0.05) $y$\\
(0.99*0.05) $z$\\
(0.02*0.34) $W(x,y)$\\
(0.21*0.34) $V(x,y)$\\
(0.4*0.34) $z$\\
(0.6*0.34) $x$\\
(0.79*0.34) $W(y,z)$\\
(0.99*0.34) $V(y,z)$\\
(0.02*0.64) $W(x,y)$\\
(0.21*0.64) $W(V(x,y),z)$\\
(0.41*0.64) $V(V(x,y),z)$\\
(0.6*0.64) $W(x,W(y,z))$\\
(0.8*0.64) $V(x,W(y,z))$\\
(0.99*0.64) $V(y,z)$\\
(-0.01*1.02) $W (W(x,y),$\\
(0.02*0.95) $W(V(x,y),z) )$\\
(0.19*1.02) $V (W(x,y),$\\
(0.22*0.95) $W(V(x,y),z) )$\\
(0.41*0.95) $V(V(x,y),z)$\\
(0.6*0.95) $W(x,W(y,z))$\\
(0.77*1.02) $W(V(x,W(y,z))$\\
(0.82*0.95) $,V(y,z))$\\
(0.99*1.02) $V(V(x,W(y,z))$\\
(1.04*0.95) $,V(y,z))$\\
  \endSetLabels
\strut\AffixLabels{\includegraphics*[scale=0.5, width=120mm]{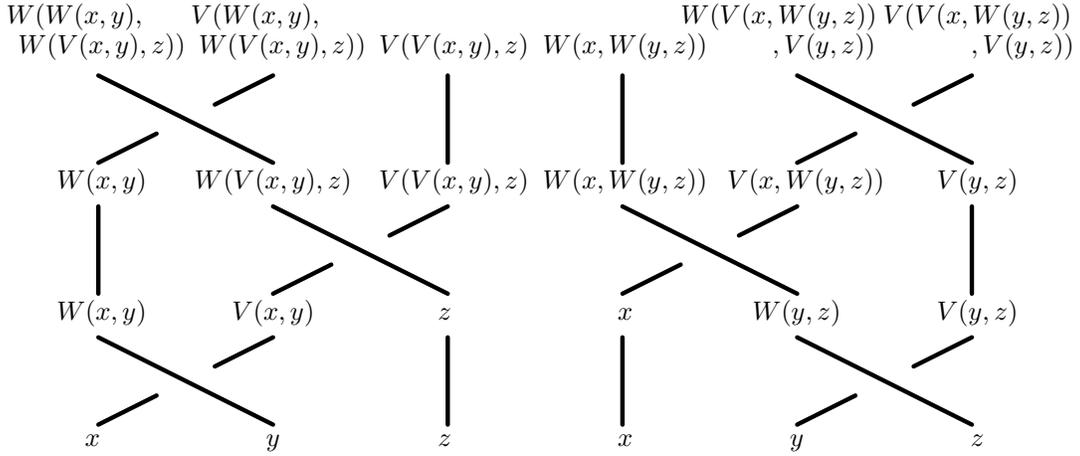}}
 \caption{Diagrammatic expression of free-product group-theoretical Yang-Baxter Equation}
 \label{fig:axiom}
  \end{center}
\end{figure}

There are two symmetries in solutions of FGYBE or Wada-type representations.
If $(W,V)$ is a solution of FGYBE, then the {\em dual solution} $(\barW,\barV)$, defined by
\[ \barW(x,y) = V(y,x),\;\;\; \barV(x,y)=W(y,x)\]
is also a solution of FGYBE. 
We remark that in \cite{w}, Wada mentioned that the dual solution is derived from the symmetry on the free group involution $\iota: F_{2} \rightarrow F_{2}$ given by $x \mapsto x^{-1}$ and $y \mapsto y^{-1}$, but it is not true. The dual solution is rather derived from the symmetry of the FGYBE under the reversal involution
$\textrm{rev}: \Z*\Z*\Z \rightarrow  \Z*\Z*\Z$ 
 given by $(x,y,z) \mapsto (z,y,x)$. 
In fact, the dual solution still exists for the similar Yang-Baxter equation in the monoidal category of monoids with monoidal structure given by the free product.

For a Wada-type representation there is an additional symmetry. Let $(W^{r},V^{r})$ be a pair of reduced words such that $\tau_{W,V}^{-1} = \tau_{W^{r}, V^{r}}$. Then $(W^{r},V^{r})$ also defines a Wada-type representation, hence it is also a solution of FGYBE. We call this solution $(W^{r},V^{r})$ the {\em inverse solution} of $(W,V)$: This symmetry is derived from the involution $\iota: B_{3} \rightarrow B_{3}$ defined by $\iota(\sigma_{i}) = \sigma_{i}^{-1}$.

Now we state the main result of this paper. We will denote the empty word by $1$.

\begin{thm}[Classification of the solutions of FGYBE]
\label{thm:main}
A pair of reduced words $(W,V)$ on $\{x^{\pm1}, y^{\pm 1}\}$ is a solution of FGYBE if and only if $(W,V)$ or its dual appear in the following list.
\begin{enumerate}
\item $(W,V)=(1,x^{m})$ $(m \in \Z)$.
\item $(W,V)=(1,y)$.
\item $(W,V)=(1,xy)$.
\item $(W,V)=(x,x^{m})$ $(m \in \Z)$.
\item $(W,V)=(x,y)$.
\item $(W,V)=(y,x^{-1})$.
\item $(W,V)=(y,y^{-m}xy^{m})$ $(m \in \Z)$.
\item $(W,V)=(y, yx^{-1}y)$.
\item $(W,V)=(y^{-1},x^{-1})$.
\item $(W,V)=(y^{-1},yxy)$.
\item $(W,V)=(xy^{-1}x^{-1},xy^{2})$
\item $(W,V)=(x^{-1}y^{-1}x, y^{2}x)$.
\item $(W,V)=(y^{s},x^{m})$ $(s,m \in \Z)$.
\end{enumerate}
\end{thm}

In the above list, the solutions (5)--(12) and (13) for $s, m \in \{ \pm 1\}$ are invertible. For each solution (5)--(10), the inverse solution is equal to its dual solution (For (5) and (9), the dual solution is equal to the original solution). The inverse solution of (11) is the dual solution of (12) and vice versa. 
Thus, if we restrict our attention to Wada-type representations, we get the same list of Wada-type representations in \cite{w}:

\begin{cor}[Classification of Wada-type representation]
Up to the dual and the inverse symmetries, there are seven types of Wada-type representations, listed as (5)--(11) in Theorem \ref{thm:main}.
\end{cor}

We also remark that $\tau_{W,V}:F_{2} \rightarrow F_{2}$ is induced from a homeomorphism of two-punctured disc if and only if $\tau_{W,V}(xy) = xy$. This condition is satisfied for (5),(7) for $m=1$, (10), and (11).
\\

\textbf{Acknowledgments.}
 This research was supported by JSPS Research Fellowships for Young Scientists. The author would like to thank his supervisor Toshitake Kohno for a helpful conversation. 
He also would like to thank the referee whose comments led the author to recognize the importance of Lemma \ref{lem:free} which was implicit in the previous version without proof.

\section{Classification}

Our proof of the classification theorem is rather elementary and combinatorial, and is based on a careful analysis on cancellations of subwords.

A word on $\{x_{1}^{\pm 1},\ldots, x_{n}^{\pm 1}\}$ defines an element of $F_{n}$, the rank $n$ free group generated by $\{x_{1},\ldots,x_{n}\}$. To distinguish elements of $F_{n}$ and representing words, for two words $W$ and $V$ we write $W \equiv V$ if they are the same as words and write by $W=V$ if they represent the same element of $F_{n}$. 
A word $W$ is {\em reduced} if $W$ contains no subwords of the form $x_{i}^{\pm 1} x_{i}^{\mp1}$. An operation deleting the subword of the form $x_{i}^{\pm 1} x_{i}^{\mp1}$ is called a {\em reducing operation}.

By applying reducing operations repeatedly, every word $W$ is changed to the unique reduced word $\red(W)$ which satisfies $W=\red(W)$. We say a subword $A$ in $W$ is {\em not canceled} in $W$ if for any sequences of reducing operations 
\[ W \rightarrow W_{1} \rightarrow \cdots \rightarrow \red(W) \]
each reducing operation $W_{i} \rightarrow W_{i+1}$ preserves the subword $A$. That is, each reducing operation does not involve the letters in the subword $A$.

In our proof of the classification theorem, we will frequently use the following simple fact.

\begin{lem}
\label{lem:fund}
Let $W$ and $V$ be words on $\{x_{1}^{\pm 1},\ldots, x_{n}^{\pm 1}\}$ such that $W=V$. 
Let $A \equiv x_{1}^{\pm 1}\cdots $ be a common subword in $W$ and $V$. Thus, $W \equiv W' A W''$ and $V \equiv V' A V''$  for some words $W',W'',V'$ and $V''$. 
Assume that $W'$ and $V'$ contain no $x_{1}^{\pm 1}$ and a subword $A$ is not canceled in both $W$ and $V$. Then $W' = V'$.
\end{lem}

Now we start to prove the classification theorem. 
Let $W=W(x,y)$ and $V=V(x,y)$ be reduced words on $\{x^{\pm 1}, y^{\pm 1}\} $.
First we analyze a generic case: we assume that $W(x,y)$ contains at least one $y^{\pm 1}$ and $V(x,y)$ contains at least one $x^{\pm 1}$. The remaining case, the case $W(x,y) \equiv x^{m}$ or $V(x,y) = y^{m}$ $(m \in \Z)$ will be treated later.

In the analysis of a generic case, the following lemma plays an important role.
This lemma will be used to show certain subwords never cancel, and allows us to apply Lemma \ref{lem:fund}.

\begin{lem}
\label{lem:free}
Let $(W,V)$ be a solution of FGYBE. Assume that $W(x,y)$ contains at least one $y^{\pm 1}$ and $V(x,y)$ contains at least one $x^{\pm 1}$. Then $W$ and $V$ generates the free group of rank two.
\end{lem}
\begin{proof}
Assume that $W$ and $V$ do not generate a free group of rank two. Since $W$ and $V$ are non-trivial elements, the subgroup generated by $W$ and $V$ is an infinite cyclic group generated by a reduced word $A(x,y)$.
Thus $W(x,y)=A(x,y)^{w}$ and $V(x,y)=A(x,y)^{v}$ for some non-zero integers $w,v \in \Z$.
Since we have assumed that $W$ contains at least one $x^{\pm 1}$ and $V$ contains at least one $y^{\pm1}$, the reduced word $A(x,y)$ contains both $x^{\pm 1}$ and $y^{\pm 1}$.

In this situation the three conditions [T],[M] and [B] are written as 
\begin{align*}
\text{[T]} & \quad A( A(x,y)^{w},A(A(x,y)^{v},z)^{w})^{w} = A(x,A(y,z)^{w})^{w}. \\
\text{[M]} & \quad A(A(x,y)^{w}, A(A(x,y)^{v},z)^{w})^{v} = A(A(x,A(y,z)^{w})^{v}, A(y,z)^{v})^{w}. \\
\text{[B]}  & \quad  A(A(x,y)^{v},z)^{v} = A(A(x,A(y,z)^{w}), A(y,z)^{v})^{v}.
\end{align*}

In the free group $F_{3}$, the equation $f^{n}=g^{n} (n\neq 0)$ implies $f=g$, hence by [T] and 
[B], we get
\[ A(A(x,y)^{w}, A(A(x,y)^{v},z)) = A(x,A(y,z)^{w}) \]
and
\[ A(A(x,y)^{v},z) =  A(A(x,A(y,z)^{w}), A(y,z)^{v}). \]
Therefore by [M] we get an equality
\begin{equation} 
\label{eqn:cyc}
A(x,A(y,z)^{w})^{v}= A(A(x,y)^{v},z)^{w}.
\end{equation}

The key observation is the following.
\begin{claim}
\label{claim:cyc}
If a reduced word $A(x,y)$ satisfies the equation {\rm (\ref{eqn:cyc})}, then the word
$A(x,y)$ is cyclically reduced.
\end{claim}
\begin{proof}
Let us write 
\[ A(x,y) \equiv R(x,y) C(x,y) R(x,y)^{-1}\]
 where $C(x,y)$ is cyclically reduced. We show $R(x,y)$ must be the empty word.
Let $m$ be the number of letters $x^{\pm 1}$ appearing in $R(x,y)$.
Observe that
\begin{eqnarray*}
A(x, A(y,z)^{w})^{v}\!\! & \! = & \! R(x, A(y,z)^{w})C(x, A(y,z)^{w})^{v}R(x, A(y,z)^{w})^{-1} \\
\!\! &\! =  & \!R(x, A(y,z)^{w}) C(x, R(y,z)C(y,z)^{w}R(y,z)^{-1})^{v} R(x, A(y,z)^{w})^{-1}.
\end{eqnarray*}

Let us study the reducing procedure of $C(x, R(y,z)C(y,z)^{w}R(y,z)^{-1})^{v}$.
Since $C(x,y)$ is cyclically reduced, a cancellation can occur only between the letters appearing in the subwords $R(y,z)$ and $R(y,z)^{-1}$: all cancellations are regarded as a subword cancellation of the form
\[ R(y,z)^{\pm 1} R(y,z)^{\mp 1} \rightarrow 1. \]
Thus unless $C(x,y) \equiv y^{\pm 1} \cdots y^{\pm 1} $, $\red( C(x, R(y,z)C(y,z)^{w}R(y,z)^{-1})^{v} )$ is cyclically reduced. 
If $C(x,y) \equiv y^{\pm 1} \cdots y^{\pm 1} $, then 
\[ \red (C(x, R(y,z)C(y,z)^{w}R(y,z)^{-1})^{v}) \equiv R(y,z)\underline{C(y,z)^{\pm w} \cdots C(y,z)^{\pm w}} R(y,z)^{-1}\]
and the underlined subword $\underline{C(y,z)^{\pm w} \cdots C(y,z)^{\pm w}} $ is cyclically reduced.
So we conclude that if we write 
\[ \red (A(x, A(y,z)^{w})^{v}) \equiv R_{1} C_{1}R_{1}^{-1}\]
 where $C_{1}$ is a cyclically reduced subword, then $R_{1} = R(x, A(y,z)^{w})$ or $R_{1} = R(x, A(y,z)^{w})R(y,z)$. The latter case occurs only if $C(x,y) \equiv y^{\pm 1} \cdots y^{\pm 1} $.
In particular, the number of $x^{\pm 1}$ in $R_{1}$ is $m$.

By similar arguments for $A(A(x,y)^{v},z)^{w}$, if we write 
\[ \red (A(A(x,y)^{v},z)^{w}) \equiv R_{2} C_{2}R_{2}^{-1}\]
 where $C_{2}$ is a cyclically reduced subword, then $R_{2} = R(A(x,y)^{v}, z)$ or $R_{2} = R(A(x,y)^{v}, z)R(x,y)$. The latter case occurs only if $R(x,y) \equiv x^{\pm 1} \cdots x^{\pm 1}$.

By equation (\ref{eqn:cyc}), $R_{1} C_{1}R_{1}^{-1} \equiv R_{2} C_{2}R_{2}^{-1}$ so $R_{1} \equiv R_{2}$ and $C_{1} \equiv C_{2}$. We show $m=0$ by counting the number of $x^{\pm 1}$ in the word $R_{2}$.

First we treat the case $R_{2} = R(A(x,y)^{v}, z)$. Assume that $m \neq 0$, so $R(x,y)$ is written as $R(x,y) \equiv \cdots x^{a} \cdots $ $(a \neq 0)$.
Then 
\[ R_{2} \equiv \red( R(A(x,y)^{v},z) ) \equiv \cdots R(x,y)C(x,y)^{av}R(x,y)^{-1} \cdots, \] 
that is, a subword $x^{a}$ in $R(x,y)$ yields a subword $ R(x,y)C(x,y)^{av}R(x,y)^{-1}$ in $R_{2}$. In the subword $R(x,y)$ and $R(x,y)^{-1}$ the letter $x^{\pm 1}$ appears $m$ times respectively, so in the whole word $R_{2}$ the letter $x^{\pm 1}$ appears at least $2m$ times.

Next we treat the case $R_{2} = R(A(x,y)^{v}, z)R(x,y)$. 
Recall that this happens only if $C(x,y) \equiv x^{\pm 1} \cdots x^{\pm 1}$. In particular, $C(x,y)$ contains at least one $x^{\pm 1}$.
As we have seen in the previous case, a subword $x^{a}$ in $R(x,y)$ yields a subword $ R(x,y)C(x,y)^{av}R(x,y)^{-1}$ in $\red(R(A(x,y)^{v}, z))$. In the present case, 
\[ R_{2} \equiv \red (R(A(x,y)^{v}, z)R(x,y)) \equiv \red ( \red (R(A(x,y)^{v}, z) \cdot R(x,y))  \]
hence the last subword $R(x,y)^{-1}$ in $R(x,y)C(x,y)^{av}R(x,y)^{-1}$, which is also a subword of $\red ( R(A(x,y)^{v}, z) )$, might be canceled in the word $R_{2}$. This shows that $R_{2}$ contains at least one subword of the form $R(x,y)C(x,y)^{av}$. 
 Since in this case $C(x,y)$ contains at least one $x^{\pm 1}$, the subword $R(x,y)C(x,y)^{av}$ contains at least $(m+1)$ $x^{\pm 1}$. In particular, $R_{2}$ contains at least $(m+1)$ $x^{\pm 1}$.
 
On the other hand, we have observed that $R_{1}$ contains exactly $m$ $x^{\pm 1}$. This contradicts $R_{1} \equiv R_{2}$. Hence $m=0$, and we conclude that $R(x,y)$ contains no $x^{\pm 1}$.

A similar argument shows that $R(x,y)$ contains no $y^{\pm 1}$ as well, so $R(x,y)$ must be the empty word. Thus $A(x,y) \equiv C(x,y)$ so $A(x,y)$ is cyclically reduced.
\end{proof}

By using this observation we complete the proof of lemma.
Assume that the number of $x^{\pm 1}$ and $y^{\pm 1}$ in $A(x,y)$ are $a$ and $b$, respectively.
Since we have observed that $A(x,y)$ is cyclically reduced, 
\[ \red( A(x, A(y,z)^{w})^{v} ) \equiv A(x, A(y,z)^{w})^{v}. \]
Thus, the number of $x^{\pm 1}$ in $\red(A(x, A(y,z)^{w})^{v})$ is $a|v|$. Similarly, 
\[ \red( A(A(x,y)^{v},z)^{w}) \equiv A(A(x,y)^{v},z)^{w} \]
so it contains $a^{2}|v| |w|$ $x^{\pm 1}$. By equation (\ref{eqn:cyc}), $a|v| = a^{2}|v||w|$. By considering the letter $z^{\pm 1}$, we get $b|w| = b^{2}|v||w|$. By hypothesis, $a, b, v,w \neq 0$ so we get $a=b=1$. 

However, the words $A(x,y) = x^{\pm 1} y^{\pm 1}, y^{\pm 1}x^{\pm 1}$ never satisfy the three equations [T],[M] and [B]. This shows that $W$ and $V$ generates a free group of rank two.
\end{proof} 

By using Lemma \ref{lem:free} we show that a solution of FGYBE must satisfy other equalities which are simpler than [T],[M] and [B]. First we consider the case $W \equiv x^{p}y^{s}\cdots $ $(s>0)$.

\begin{lem}
\label{lem:1plus}
Let $(W,V)$ be a solution of FGYBE. Assume that $W\equiv x^{p}y^{s}\cdots $ $(s>0)$, and that $V(x,y)$ contains at least one $x^{\pm 1}$. Then,
\begin{enumerate}
\item $W(x,y)^{p}V(x,y)^{p}=x^{p}y^{p}$. 
\item If $s \neq 1$, $(W,V) \equiv (y^{s},x^{m})$.
\end{enumerate}
\end{lem}
\begin{proof}
By hypothesis, we have 
\begin{eqnarray*}
 W( W(x,y),W(V(x,y),z) )  &= &W( W(x,y), V(x,y)^{p}z^{s} \cdots) \\
 & = & W(x,y)^{p}(V(x,y)^{p} z^{s} \cdots)^{s} \cdots \\
 & = & W(x,y)^{p} V(x,y)^{p} z^{s} \cdots 
 \end{eqnarray*}
 and
 \[ W(x, W(y,z)) = W(x, y^{p}z^{s} \cdots ) = x^{p}(y^{p}z^{s}\cdots)^{s} \cdots = x^{p}y^{p}z^{s} \cdots. \]
 Hence by [T], we get an equation
 \[ W(x,y)^{p} V(x,y)^{p} z^{s} \cdots =  x^{p}y^{p}z^{s} \cdots. \]
The left side of the above equation is written as $W(x,y)^{p} V(x,y)^{p} z^{s} R(W,V) z^{s'}\cdots $, where $R(W,V)$ is a reduced, non-empty word on $W=W(x,y)$ and $V=V(x,y)$.
By Lemma \ref{lem:free} we have observed that $W$ and $V$ generates the free group of rank two, hence $R(W,V)$ is a non-trivial element in $F_{2}$. In particular, as a word on $\{x^{\pm 1},y^{\pm 1}\}$, the word $R(W,V)$ cannot be reduced to the empty word. Therefore the subword $z^{s}$ is not canceled in the reducing procedure. Similarly, in the right side the subword $z^{s}$ is not canceled. 
So by Lemma \ref{lem:fund}, $W(x,y)^{p}V(x,y)^{p}=x^{p}y^{p}$.

To show (2), first we consider the case $W(x,y) \equiv x^{p}y^{s}x^{-p}$. Then by (1),
$x^{p}y^{ps}x^{-p}V(x,y)^{p} = x^{p}y^{p} $ so $V(x,y)^{p} = x^{p}y^{p(1-s)}$. If $|p|>1$ then $x^{p}y^{p(1-s)}$ is primitive, which contradicts $V(x,y)^{p} = x^{p}y^{p(1-s)}$. Thus $p= \pm 1, 0$.

If $p=0$, then $W(x,y) \equiv y^{s}$. By [M], $V(y^{s},z^{s}) = V(y,z)^{s}$. This equation is clearly satisfied for $s=1$, so assume that $s>1$.
If $V(y,z) \equiv y^{a}z^{b} \cdots$ $(a,b \neq 0)$, then $\red(V(y^{s},z^{s})) \equiv y^{as}z^{bs}\cdots$ but $\red(V(y,z)^{s}) \equiv y^{a}\cdots$, so this cannot happen. Similarly, $V(y,z) \equiv z^{a}y^{b}\cdots$ is also impossible, hence $V(y,z)\equiv y^{m}, z^{m}$ $(m \in \Z)$. Since we have assumed that $V(x,y)$ contains at least one $x^{\pm 1}$, we conclude $(W,V) \equiv (y^{s},x^{m})$ if $s \neq 1$.

If $p=1$, then $W(x,y) \equiv xy^{s}x^{-1}$ and $V(x,y) \equiv xy^{1-s}$. By [B],
\[ xy^{1-s}z^{1-s} = xyz^{s(1-s)}y^{-1}(yz^{1-s})^{1-s}. \]
This equality is satisfied only if $s=1$.
Similarly, if $p=-1$ then by [B] 
\[ z^{1-s}y^{1-s}x = (z^{1-s}y)^{s}y^{-1}z^{s(1-s)}y x. \]
 This equality cannot be satisfied for any $s>0$.
This completes the proof of (2) for the case $W(x,y) \equiv x^{p}y^{s}x^{-p}$.

Next we assume that $W(x,y) \not \equiv x^{p}y^{s}x^{-p}$.
Then for $s>1$, $\red(W(x,y)^{s}) \not \equiv x^{a}y^{b}, x^{a}y^{b}x^{c}$ for any $a,b,c \in \Z -\{0\}$, so 
\[ \red( W(x,y)^{s}) \equiv x^{p}y^{s}x^{q}y^{r} \cdots\]
 for some $q,r \neq 0$.
Therefore by [T], we have
\[ W(x,y)^{p} V(x,y)^{p} z^{s} V(x,y)^{q} z^{r} \cdots = x^{p}y^{p}z^{s}y^{q}z^{r} \cdots \] 
The subwords $z^{s}$ and $z^{r}$ are not canceled in both sides, so by Lemma \ref{lem:fund}, $V(x,y)^{q} = y^{q}$. Since $q \neq 0$, we get $V(x,y)=y$, which contradicts the hypothesis. 
\end{proof}

\begin{prop}
\label{prop:plus}
Let $(W,V)$ be a solution of FGYBE. Assume that $W\equiv x^{p}y^{s}\cdots $ $(s>0)$ and that $V(x,y)$ contains at least one $x^{\pm 1}$. Then $W(x,y) \equiv x^{p}yx^{q}$ or $W(x,y)=y^{s}$.
\end{prop}
\begin{proof}
We have already seen that if $s\neq 1$, then $(W,V)=(y^{s},x^{m})$, so we assume that $s=1$.
Assume that $W$ is written as $W(x,y) \equiv x^{p}yx^{q}y^{r}\cdots$ for $q,r \neq 0$.
By [T], we have
\[ W(x,y)^{p} (V(x,y)^{p} z V(x,y)^{q} z^{r} \cdots) \cdots  = x^{p}(y^{p}zy^{q}z^{r} \cdots) \cdots \]
Since $s=1$, the subwords $z$ and $z^{r}$ are not canceled so $V(x,y)^{q} = y^{q}$. Then $q\neq 0$ implies $V(x,y)=y$, which contradicts the hypothesis.
\end{proof}

Next we study the case $W \equiv x^{p}y^{-s}\cdots $ $(s>0)$. The argument is similar, but we need to consider the initial part of the word $W^{-1}$ so we must take care of the suffix of $W$ as well.

\begin{lem}
\label{lem:1minus}
Let $(W,V)$ be a solution of FGYBE. Assume that $W\equiv x^{p}y^{-s}\cdots  \equiv \cdots y^{-r}x^{-q}$ $(s>0, r\neq 0)$, and that $V(x,y)$ contains at least one $x^{\pm 1}$. Then,
\begin{enumerate}
\item $W(x,y)^{p}V(x,y)^{q}=x^{p}y^{q}$.
\item if $s \neq 1$, $(W,V) \equiv (y^{-s},x^{m})$.
\end{enumerate}
\end{lem}
\begin{proof}
By hypothesis $W(x,y)^{-1} \equiv x^{q}y^{r}\cdots$, so by the same argument as in Lemma \ref{lem:1plus}, we get an equation
\[ W(x,y)^{p}V(x,y)^{q}z^{r} \cdots = x^{p}y^{q}z^{r}\cdots \]
from [T].
By a similar argument as in Lemma \ref{lem:1plus}, in both sides the subwords $z^{r}$ are not canceled, so we conclude $W(x,y)^{p}V(x,y)^{q}=x^{p}y^{q}$.

To show (2), first we consider the case $W(x,y) \equiv x^{p}y^{-s}x^{-p}$. Then by a similar argument of the proof of Lemma \ref{lem:1plus}, we conclude $s=1$ unless $(W,V)=(y^{-s},x^{m})$.
Next we assume that $W(x,y) \not \equiv x^{p}y^{-s}x^{-p}$. If $s>1$, then 
\[ \red(W^{-s}) \equiv x^{q}y^{r}x^{q'}y^{r'} \cdots \]
 for $q',r' \neq 0$. By [T], 
\[ W(x,y)^{p} V(x,y)^{q} z^{r} V(x,y)^{q'} z^{r'}  \cdots  = x^{p}y^{q}z^{r}y^{q'}z^{r'} \cdots. \]
In both sides the subwords $z^{r'}$ and $z^{r}$ are not canceled, so $V(x,y)^{q'} = y^{q'}$. $q' \neq 0$ implies $V(x,y)=y$, which contradicts the hypothesis.
\end{proof}

The following Proposition is proved in the same way as Proposition \ref{prop:plus}. 

\begin{prop}
\label{prop:minus}
Let $(W,V)$ be a solution of FGYBE. Assume that $W\equiv x^{p}y^{-s}\cdots \equiv \cdots y^{-r}x^{-q}$ $(s>0, r\neq 0)$, and that $V(x,y)$ contains at least one $x^{\pm 1}$. Then $W(x,y) \equiv x^{p}y^{-1}x^{q}$ or $W(x,y)=y^{-s}$.
\end{prop}

Summarizing, we obtain candidates of general solutions of FGYBE.

\begin{prop}
\label{prop:generic}
Let $(W,V)$ be a solution of FGYBE. Assume that $W$ contains at least one $y^{\pm 1}$ and $V$ contains at least one $x^{\pm 1}$. Moreover, assume that $(W,V) \neq (y^{s},x^{m})$. Then,

\begin{enumerate}
\item $W(x,y) \equiv x^{p}y^{\varepsilon}x^{q}$, $V(x,y) \equiv y^{r}x^{\delta}y^{s}$, where $p,q,r,s \in \Z$ and $\varepsilon,\delta \in \{\pm 1\}$.
\item $pqrs=0$. 
\end{enumerate}

\end{prop}
\begin{proof}
(1) for $W(x,y)$ follows from Proposition \ref{prop:plus} and Proposition \ref{prop:minus}. To prove (1) for $V(x,y)$ we just consider the dual solution.

To show (2), we observe that by [M]
\[ [(y^{r}\underline {x^{\delta}}\, y^{s})^{p} \underline{z^{\varepsilon}}\,(y^{r}x^{\delta}y^{s})^{q} ]^{r} \cdots = [ (y^{p}\underline{z^{\varepsilon}}\,y^{q})^{r}\underline{x^{\delta} }\,(y^{p}z^{\varepsilon}y^{q})^{s}]^{p}\cdots .\]
The desired equation $pqrs = 0$ is obtained as follows.
As in the argument in Lemma \ref{lem:1plus}, underlined subwords $\underline{x^{\delta}}$ and $\underline{z^{\varepsilon}}$ are not canceled during the reducing procedure. By looking at non-canceled subwords of the form $x^{\pm 1}$ or $z^{\pm 1}$ derived from the underlined subwords, we get the relation of the form $W(x,y)^{N} = y^{N}$ or $V(y,z)^{N} = y^{N}$ for some $N \in \{p,q,r,s\}$. If $pqrs \neq 0$, then $N\neq 0$, so this contradicts the hypothesis.

To illustrate a precise argument, let us consider the case $r>0$ and $p>0$, for example.
In this case, we get
\[ y^{r}x^{\delta} \cdots = (y^{p} z^{\varepsilon}y^{q})^{r} x^{\delta}\cdots. \]
As in the argument in Lemma \ref{lem:1plus}, the subwords $x^{\delta}$ are not canceled in both sides, hence we get $y^{r}=(y^{p} z^{\varepsilon}y^{q})^{r} = W(y,z)^{r}$. Since $r \neq 0$ this implies $W(y,z)=y$, which is a contradiction.

Other cases are proved in a similar way.
\end{proof}

Now we are ready to give a classification of generic solutions of FGYBE.
\begin{prop}
Let $(W,V)$ be a solution of FGYBE. Assume that $W$ contains at least one $y^{\pm 1}$ and $V$ contains at least one $x^{\pm 1}$. Then $(W,V)$ or its dual is one of the solutions (6)--(13) listed in Theorem \ref{thm:main}.
\end{prop}
\begin{proof}
By Proposition \ref{prop:generic} if $(W,V)$ is not the solution (13), that is, if $(W,V) \neq (y^{s},x^{m})$, then
$W$ and $V$ are written as
\[ W(x,y) \equiv x^{p}y^{\varepsilon} x^{q}, V(x,y) \equiv y^{r}x^{\delta}y^{s}\; \; (pqrs = 0, \;\;\varepsilon, \delta \in \{\pm 1\}) .\]
By considering the dual solution if necessary, we may assume that either $r=0$ or $s=0$. \\

\noindent
\textbf{Case 1:} $r=0$, $s \neq 0$.\\

By [B], we have an equality
\[ (x^{\delta}y^{s})^{\delta} z^{s}= [x^{\delta}(y^{p}z^{\varepsilon}y^{q})^{s}]^{\delta} (y^{\delta}z^{s})^{s}\]

Assume that $\delta= -1$. Then the above equality is reduced to the equality
\[ y^{-s} = (y^{p}z^{\varepsilon}y^{q})^{-s}. \] 
This equation is satisfied only if $s=0$, which is a contradiction.
Therefore $\delta = +1$ and we get an equation
\begin{equation}  
\label{eqn:eqn1}
y^{s} z^{s}= (y^{p}z^{\varepsilon}y^{q})^{s}(yz^{s})^{s}. 
\end{equation}

\begin{claim}
\label{claim:eqn}
If $s \neq 0$, then the equation (\ref{eqn:eqn1}) is satisfied only if $(p,\varepsilon,q,s)=(1,-1,-1,2)$.
\end{claim}
\begin{proof}[Proof of Claim]
First we treat the case $s>0$. Then
\[ y^{s}z^{s} = \overbrace{y^{p} \cdots z^{\varepsilon} \underline{y^{q}}}^{(y^{p}z^{\varepsilon}y^{q})^{s}} \overbrace{\underline{y} z^{s} \cdots}^{(yz^{s})^{s}} \]
In the right side, reducing can happen only at the underlined subword $\underline{y^{q} y}$, hence $q=-1$.
If $s=1$, then the equation (\ref{eqn:eqn1}) is written as $y^{s}z^{s} = y^{p}z^{\varepsilon + s}$, which is impossible since $\varepsilon \neq 0$.
Thus, $s>1$ and the equation (\ref{eqn:eqn1}) is now written as
\[ y^{s}z^{s} = y^{p} \cdots z^{\varepsilon}  \underline{y^{p} y^{-1}} \underline{ z^{\varepsilon}z^{s}} y z^{s} \cdots.\]
A reducing is possible only at the underlined subwords $\underline{y^{p} y^{-1}} $ or $\underline{ z^{\varepsilon}z^{s}}$. Since $s \neq 1$, $z^{\varepsilon +s}$ is non-trivial, so $p=1$.
Finally, the equation (\ref{eqn:eqn1}) is reduced to the equation
\[ y^{s}z^{s} = y z^{s \varepsilon +s} (yz^{s})^{s-1}. \] 
Then $s \varepsilon +s = 0$ hence $\varepsilon = -1$, and $s=2$.

Next we treat the case $s<0$.
Then the equation (\ref{eqn:eqn1}) is written as
\[ y^{s}z^{s} = \overbrace{y^{-q}z^{-\varepsilon} \cdots \underline{y^{-p}}}^{(y^{p}z^{\varepsilon}y^{q})^{s}} \overbrace{z^{-s} y^{-1} \cdots}^{(yz^{s})^{s}}. \]
If $p \neq 0$, then the right side is reduced, which is impossible. Hence $p=0$ and the equation (\ref{eqn:eqn1}) is reduced to
\[ y^{s}z^{s} = y^{-q} \cdots \underline{ z^{-\varepsilon}z^{-s} } y^{-1} \cdots. \]
In the right side, a reducing is possible only at the underlined subword  $\underline{ z^{-\varepsilon}z^{-s} }$, hence $s=-1$ and $\varepsilon = 1$. However, this leads to an equality $ y^{-1}z^{-1} = y^{-q-1}$ which is impossible. Hence there are no solution of the equation (\ref{eqn:eqn1}) for $s<0$.
\end{proof}

By Claim \ref{claim:eqn}, in this case we obtain the solution (11),
\[ W(x,y)=xy^{-1}x^{-1},\;\;\; V(x,y)= xy^{2}.\]

\noindent
\textbf{Case 2:} $r \neq 0$, $s = 0$.\\

By the similar arguments as Case 1, we conclude $\delta=+1$ and by [B] we get an equation
\[ z^{r}y^{r}=(z^{r}y)^{r}(y^{p}z^{\varepsilon}y^{q})^{r} .\]
As in Claim \ref{claim:eqn}, this equality holds only if $r=0$ or $(p,\varepsilon,q,r)= (-1,-1,1,2)$. So in this case we obtain the solution (12), 
\[ W(x,y)=x^{-1}y^{-1}x,\;\;\; V(x,y) =y^{2}x. \]

\noindent
\textbf{Case 3: $r=s=0$}.\\

\noindent
{\it Subcase 3-1:} $(\varepsilon,\delta) = (+1,+1)$.\\

By Lemma \ref{lem:1plus}, $(x^{p}yx^{q})^{p} x^{p} = x^{p}y^{p}$.
This equation is satisfied only if $p=0$ or $p=-q$. If $p=0$, then by direct computation we get $q=0$.  So we obtain the dual of the solution (7),
\[ W(x,y)= x^{p}yx^{-p}, \;\;\; V(x,y)=x. \]
{}

\noindent
{\it Subcase 3-2:} $(\varepsilon,\delta) = (-1,+1)$.\\

By Lemma \ref{lem:1minus}, $(x^{p}y^{-1}x^{q})^{p} x^{-q} = x^{p}y^{-q}$.
This equation is satisfied only if $p=q=0$ or $p=q=1$, so we obtain the dual of the solutions (6) and (8),
\[ W(x,y)=y^{-1},\;  V(x,y)=x\; \textrm{ or } W(x,y)=xy^{-1}x,\;  V(x,y)=x. \]
{}

\noindent
{\it Subcase 3-3:} $(\varepsilon,\delta) = (+1,-1)$.\\

By [M], we have $x^{-q}yx^{-p} = W(x,y)^{-1} = W(x^{-1},y^{-1}) = x^{-p}yx^{-q}$, hence $p=q$.
By Lemma \ref{lem:1plus}, $(x^{p}yx^{p})^{p} x^{-p} = x^{p}y^{p}$.
This equation is satisfied only if $p=1$ or $p=0$. Hence in this case we obtain the solution (6) and the dual of (10),
\[ W(x,y)=  y, \;\;\; V(x,y)=x^{-1}\;  \textrm{ or } W(x,y)=xyx,\;  V(x,y)=x^{-1}. \]
{}

\noindent
{\it Subcase 3-4:} $(\varepsilon,\delta) = (-1,-1)$.\\

By [M], we have $x^{-q}yx^{-p} = W(x,y)^{-1} = W(x^{-1},y^{-1}) = x^{-p}yx^{-q}$, hence $p=q$.
By Lemma \ref{lem:1minus}, $(x^{p}y^{-1}x^{p})^{p} x^{p} = x^{p}y^{-p}$.
This equation is satisfied only if $p=0$, so we obtain the solution (9),
\[ W(x,y)= y^{-1}, \;\;\; V(x,y)=x^{-1}. \]
\end{proof}

Next we study the case that $W(x,y)$ contains no $y^{\pm 1}$ or $V(x,y)$ contains no $x^{\pm1}$,  which was excluded in the previous arguments.
By considering the dual solution if necessary, we may assume that $W(x,y) \equiv x^{m}$ $(m \in \Z)$.
By [T], it is easily checked that $(W \equiv x^{m},V)$ is a solution of FGYBE only if $m= 0, 1$. First we study the case $m=0$.

\begin{prop}
\label{prop:Wisphi}
Let $(W,V)$ be a solution of FGYBE such that $W \equiv 1$. Then $(W,V)$ or its dual is one of the solutions (1)--(3) listed in Theorem \ref{thm:main}.
\end{prop}

\begin{proof}
First of all assume that $V \equiv x^{m}$ or $V \equiv y^{m}$. Then it is directly checked that $(W,V)$ is a solution of FGYBE if and only if $(W,V) = (1,x^{m})$ $(m \in \Z)$ or $(1, y)$ so we get solutions (1) and (2).

Assume that $V$ contains both $x^{\pm 1}$ and $y^{\pm 1}$ and put
\[ V(x,y) \equiv x^{a_{0}}y^{a_{1}}\cdots x^{a_{2n-2}}y^{a_{2n-1}}x^{a_{2n}} \]
where $a_{1}, a_{2n-1} \neq 0$.

First consider the case $a_{1}>0$. Then by [B],
\[ V(x,y)^{a_{0}}z^{a_{1}}V(x,y)^{a_{2}}z^{a_{3}} \cdots =  V(x, 1)^{a_{0}}y^{a_{0}}z^{a_{1}}y^{a_{2}}z^{a_{3}}\cdots \]
The subwords $z^{a_{1}}$ and $z^{a_{3}}$ are not canceled, so $V(x,y)^{a_{0}} = V(x,1)^{a_{0}}y^{a_{0}}$ and $V(x,y)^{a_{2}} = y^{a_{2}}$. If $a_{2} \neq 0$, then the equality  $V(x,y)^{a_{2}} = y^{a_{2}}$ means that $V(x,y) = y$, which is a contradiction. Therefore $a_{2}=\cdots =a_{2n} = 0$.
Moreover,  the equality $V(x,y)^{a_{0}} = V(x,1)^{a_{0}}y^{a_{0}}$ implies that $V(x,y)$ is primitive, so $a_{0} = 0,\pm 1$.  Since we have assumed that $V$ contains both $x^{\pm 1}$ and $y^{\pm 1}$, $a_{0} \neq 0$.
It is directly checked that $a_{0} = -1$ is impossible. If $a_{0}=1$ then we obtain the solution (3),
\[ W(x,y)= 1,\;\;\; V(x,y) = xy. \]

In a similar way, it is shown that there are no solutions of FGYBE if $a_{1}<0$.
\end{proof}

We complete the proof of Theorem \ref{thm:main} by studying the case $m=1$.

\begin{prop}
Let $(W,V)$ be a solution of FGYBE such that $W \equiv x$. Then $(W,V)$ or its dual is (4) or (5) in Theorem \ref{thm:main}.
\end{prop}
\begin{proof}
The proof is almost the same as the proof of Proposition \ref{prop:Wisphi}.
Assume that $V \equiv x^{m}$ or $V \equiv y^{m}$. Then it is directly checked that $(W,V)$ is a solution of FGYBE if and only if $(W,V) = (x,x^{m})$ $(m \in \Z)$ or $(W,V)= (x,y)$, so we get solutions (4) or (5).

Now we assume that $V$ contains both $x^{\pm 1}$ and $y^{\pm 1}$ and put
\[ V(x,y) \equiv x^{a_{0}}y^{a_{1}}\cdots x^{a_{2n-2}}y^{a_{2n-1}}x^{a_{2n}} \]
where $a_{1}, a_{2n-1} \neq 0$.

First consider the case $a_{1}>0$. Then by [B],
\[ V(x,y)^{a_{0}}z^{a_{1}}V(x,y)^{a_{2}}z^{a_{3}} \cdots =  V(x,y)^{a_{0}}y^{a_{0}}z^{a_{1}}y^{a_{2}}z^{a_{3}}\cdots \]
Since subwords $z^{a_{1}}$ and $z^{a_{3}}$ are not canceled, we have $V(x,y)^{a_{0}} = V(x,y)^{a_{0}}y^{a_{0}}$ and $V(x,y)^{a_{2}} = y^{a_{2}}$. Then by the same argument as Proposition \ref{prop:Wisphi}, we get $a_{0}=0$ and $a_{2}=\cdots =a_{2n} = 0$. Hence there are no solutions of FGYBE in this case.
Similarly, we conclude that there are no solutions of FGYBE in the case $a_{1}<0$.
\end{proof}


\begin{thebibliography}{1}

\bibitem{b} J. Birman,
{\em{Braids, Links, and Mapping Class Groups,}}
Annals of Math. Studies \textbf{82}, Princeton Univ. Press (1974).

\bibitem{cp} J. Crisp and L. Paris,
{\em{Representation of the braid group by automorphisms of group, invariant of links, and Garside groups,}}
Pacific J. Math, \textbf{221}, (2005), 1-27. 

\bibitem{fjk} R. Fenn, M. Jordan-Santana and L. Kauffman,
{\em{Biquandles and virtual links,}}
Topology Appl. \textbf{145}, (2004), 157--175.

\bibitem{i} T. Ito,
{\em{A functor-valued extension of knot quandles,}}
J. Math. Soc. Japan, to appear.

\bibitem{w} M. Wada,
{\em{Group invariants of links,}}
Topology, \textbf{31}, (1992), 399-406.
\end{thebibliography}
\end{document}